\documentclass[graybox]{svmult}
\usepackage{mathptmx}       % selects Times Roman as basic font
\usepackage{helvet}         % selects Helvetica as sans-serif font
\usepackage{courier}        % selects Courier as typewriter font
\usepackage{type1cm}        % activate if the above 3 fonts are
\usepackage{makeidx}         % allows index generation
\usepackage{graphicx}        % standard LaTeX graphics tool
\usepackage{multicol}        % used for the two-column index
\usepackage[bottom]{footmisc}% places footnotes at page bottom
\usepackage{amssymb}
\usepackage{amsmath}
\newcommand{\raisecomma}{\raisebox{2pt}{$,$}}
\newcommand{\raisedot}{\raisebox{2pt}{$.$}}
\makeindex             % used for the subject index
\begin{document}
\title*{Fast Computation of Bernoulli, Tangent\\ and Secant Numbers}
\titlerunning{Fast Computation of Bernoulli, Tangent and Secant Numbers}
\author{Richard P. Brent and David Harvey}
\institute{Richard P. Brent \at 
Mathematical Sciences Institute,
Australian National University, Canberra, ACT 0200, Australia.
\email{Tangent@rpbrent.com}
\and David Harvey \at
School of Mathematics and Statistics,
University of New South Wales, Sydney, NSW 2052,\\ Australia.
\email{D.Harvey@unsw.edu.au}
}
\maketitle

\abstract*{We 
consider the computation of Bernoulli, Tangent (zag), and Secant (zig or
Euler) numbers. In particular, we give asymptotically fast algorithms for
computing the first $n$ such numbers in $O(n^2(\log n)^{2+o(1)})$
bit-operations. We also give very short in-place algorithms for computing
the first $n$ Tangent or Secant numbers in $O(n^2)$ integer operations. 
These algorithms are extremely simple, and fast for moderate values of $n$.
They are faster and use less space than the algorithms of Atkinson (for 
Tangent and Secant numbers) and Akiyama and Tanigawa (for Bernoulli numbers).
}

\abstract{We 
consider the computation of Bernoulli, Tangent (zag), and Secant (zig or
Euler) numbers. In particular, we give asymptotically fast algorithms for
computing the first $n$ such numbers in $O(n^2(\log n)^{2+o(1)})$
bit-operations. We also give very short in-place algorithms for computing
the first $n$ Tangent or Secant numbers in $O(n^2)$ integer operations. 
These algorithms are extremely simple, and fast for moderate values of $n$.
They are faster and use less space than the algorithms of Atkinson (for 
Tangent and Secant numbers) and Akiyama and Tanigawa (for Bernoulli numbers).
}

\section{Introduction}
\label{sec:intro}

The {\em Bernoulli numbers} are rational numbers $B_n$ 
defined by the generating function 
\begin{equation}
\sum_{n \ge 0} B_n\, \frac{z^n}{n!} = \frac{z}{\exp(z) - 1}\,\raisedot
\label{eq:Bgenfn}
\end{equation}
Bernoulli numbers 
are of interest in number theory and are related to special values of
the Riemann zeta function (see \S\ref{sec:Bernoulli}). 
They also occur as coefficients in the
Euler-Maclaurin formula, so are relevant to high-precision computation
of special functions~\cite[\S4.5]{MCA}.

\pagebreak[4]
It is sometimes convenient to consider {\em scaled} Bernoulli numbers
\begin{equation} 
C_n = \frac{B_{2n}}{(2n)!}\,\raisecomma 
\label{eq:Cn} 
\end{equation}
with generating function
\begin{equation}
\sum_{n \ge 0} C_n\,z^{2n} = \frac{z/2}{\tanh(z/2)}\,\raisedot	
\label{eq:genC}
\end{equation}
The generating functions~(\ref{eq:Bgenfn}) and~(\ref{eq:genC})
only differ by the single term $B_1 z$, since the other odd terms vanish.

The {\em Tangent numbers} $T_n$,
and {\em Secant numbers} $S_n$, are defined by
\begin{equation}
\sum_{n > 0} T_n\, \frac{z^{2n-1}}{(2n-1)!} = \tan z,\;\;\;\;
\sum_{n \ge 0} S_n\, \frac{z^{2n}}{(2n)!} = \sec z\,.
\label{eq:TSgenfn}
\end{equation}

In this paper, which is based on an a talk given by the first author
at a workshop held to mark Jonathan Borwein's sixtieth birthday, 
we consider some algorithms for computing Bernoulli, Tangent and Secant
numbers. For background, combinatorial interpretations, and 
references, see
Abramowitz and Stegun~\cite[Ch.\ 23]{AS} (where the notation differs from
ours, e.g.\ $(-1)^nE_{2n}$ is used for our $S_n$), and
Sloane's~\cite{Sloane} sequences A000367, A000182, A000364.
 
Let $M(n)$ be the number of bit-operations required for $n$-bit integer
multiplication.  The Sch\"onhage-Strassen algorithm~\cite{SS} gives
$M(n) = O(n \log n \log \log n)$, 
and F\"urer~\cite{Furer} has recently given an improved bound
$M(n) = O(n (\log n) 2^{\log^{*}n})$.
For simplicity we merely assume that $M(n) = O(n (\log n)^{1 + o(1)})$,
where the $o(1)$ term depends on the precise algorithm used
for multiplication.  For example, if the Sch\"onhage-Strassen algorithm
is used, then the $o(1)$ term can be replaced by
$\log\log\log n/\log\log n$.

In \S\S\ref{sec:Bernoulli}--\ref{sec:tangent} we mention some relevant
and generally well-known facts concerning Bernoulli, Tangent and Secant
numbers.

Recently, Harvey~\cite{Harvey10a}
showed that the {\em single} number $B_n$ can be computed in
$O(n^2(\log n)^{2 + o(1)})$ bit-operations using a modular algorithm.
In this paper we show that {\em all}
the Bernoulli numbers $B_0, \ldots,
B_n$ can be computed with the same complexity bound (and similarly for
Secant and Tangent numbers).

In \S\ref{sec:fast} we give a relatively simple algorithm 
that achieves the slightly weaker bound
$O(n^2(\log n)^{3 + o(1)})$.
In \S\ref{sec:faster} we describe the improvement to
$O(n^2(\log n)^{2 + o(1)})$.
The idea is similar to that espoused by Steel~\cite{Steel},
although we reduce the problem to division rather than multiplication.
It is an open question whether the {\em single} number $B_{2n}$ can
be computed in $o(n^2)$ bit-operations.

In \S\ref{sec:three_term} we
give very short in-place algorithms for computing the first $n$
Secant or Tangent numbers using $O(n^2)$ integer operations.  
These algorithms are extremely simple, and fast
for moderate values of~$n$ (say $n \le 1000$), although 
asymptotically not as fast as the algorithms given
in \S\S\ref{sec:fast}--\ref{sec:faster}.
Bernoulli numbers can easily
be deduced from the corresponding Tangent numbers using the relation
(\ref{eq:TkB2k}) below.

\section{Bernoulli Numbers}
\label{sec:Bernoulli}
From the generating function (\ref{eq:Bgenfn})
it is easy to see that the $B_n$ are rational numbers, 
with $B_{2n+1} = 0$ if $n > 0$.
The first few nonzero $B_n$ are:
$B_0 = 1$, $B_1 = -1/2$, $B_2 = 1/6$, $B_4 = -1/30$,
$B_6 = 1/42$, $B_8 = -1/30$, $B_{10} = 5/66$,
$B_{12} = -691/2730$, $B_{14} = 7/6$.

The denominators of the Bernoulli numbers are given by the
{\em Von Staudt~-- Clausen Theorem}~\cite{Clausen,Staudt},
which states that
\begin{equation*} B'_{2n} := 
B_{2n} \; + \sum_{(p-1)|2n} \frac{1}{p} \; \in {\mathbb Z}\,.\end{equation*}
Here the sum is over all primes $p$ for which $p-1$ divides $2n$.

Since the ``correction'' $B'_{2n} - B_{2n}$ is easy to compute,
it might be convenient in a program to store the integers $B'_{2n}$
instead of the rational numbers $B_{2n}$ or $C_{n}$.

Euler found that the Riemann zeta-function for even
non-negative integer arguments can be
expressed in terms of 
Bernoulli numbers~-- the relation is
\index{Riemann zeta-function!Bernoulli numbers}
\begin{equation}
(-1)^{n-1}\frac{B_{2n}}{(2n)!} = \frac{2\zeta(2n)}{(2\pi)^{2n}}\,\raisedot
\label{eq:B2k_zeta}
\end{equation}
Since $\zeta(2n) = 1 + O(4^{-n})$ as $n \rightarrow +\infty$, we see
that
\begin{equation*} 
|B_{2n}| \sim \frac{2\,(2n)!}{(2\pi)^{2n}}\,\raisedot
\end{equation*}
From Stirling's approximation to $(2n)!$, the number of bits
in the integer part of $B_{2n}$ is $2n\lg n + O(n)$ (we write
$\lg$ for $\log_2$).
Thus, it takes $\Omega(n^2 \log n)$ space to store $B_1, \ldots, B_n$.
We can not expect any algorithm to compute $B_1, \ldots, B_n$ in fewer 
than $\Omega(n^2 \log n)$ bit-operations.

Another connection between the Bernoulli numbers and the Riemann
zeta-function is the identity
\begin{equation}
\frac{B_{n+1}}{n+1} = -\zeta(-n)
\label{eq:Bn1_zeta}
\end{equation}
for $n \in {\mathbb Z}$, $n \ge 1$.
This follows from (\ref{eq:B2k_zeta}) and the functional equation for
the zeta-function, or directly from a contour integral representation
of the zeta-function~\cite{Titchmarsh}.

From the generating function (\ref{eq:Bgenfn}),
multiplying both sides by $\exp(z) - 1$ and equating coefficients of $z$, 
we obtain the recurrence
\begin{equation}
\sum_{j=0}^k \binom{k+1}{j}\,B_j = 0 \;\text{for}\; k>0.
\label{eq:unstable}
\end{equation}
This recurrence has traditionally been used to compute
$B_0, \ldots, B_{2n}$ with $O(n^2)$ arithmetic operations,
for example in~\cite{Knuth62}.
However, this
is unsatisfactory if floating-point numbers are used, because the
recurrence is {\em numerically unstable}: the relative error in the
computed $B_{2n}$ is of order $4^n \varepsilon$ 
if the floating-point arithmetic
has precision $\varepsilon$, i.e.\ $\lg (1/\varepsilon)$ bits.

Let $C_n$ be defined by (\ref{eq:Cn}). Then, multiplying each side of
(\ref{eq:genC}) by $\sinh(z/2)/(z/2)$
and equating coefficients gives the recurrence
\begin{equation}
\sum_{j=0}^k \frac{C_j}{(2k+1-2j)!\,4^{k-j}} =
\frac{1}{(2k)!\;4^k}\,\raisedot
\label{MCA460}
\end{equation}
Using this recurrence to evaluate $C_0, C_1, \ldots, C_n$, the relative
error in the computed $C_n$ is only $O(n^2 \varepsilon)$, which is satisfactory
from a numerical point of view.

Equation~(\ref{eq:B2k_zeta}) can be used in several ways to compute
Bernoulli numbers.  If we want just one Bernoulli number $B_{2n}$ then
$\zeta(2n)$ on the right-hand-side of~(\ref{eq:B2k_zeta})
can be evaluated to sufficient accuracy using the Euler product:
this is the ``zeta-function'' algorithm for computing Bernoulli numbers
mentioned (with several references to earlier work) 
by Harvey~\cite{Harvey10a}.
On the other hand, if we want several Bernoulli numbers, then we can use
the generating function
\begin{equation}
\frac{\pi z}{\tanh(\pi z)} = -2 \sum_{k=0}^\infty
(-1)^{k}\zeta(2k)z^{2k}\,,
\label{eq:coth-zeta}
\end{equation}
computing the coefficients of $z^{2k}$, $k \le n$, to sufficient accuracy,
as mentioned in~\cite{BBC97,BCEM,BCS}.
This is similar to the fast algorithm that we describe in \S\ref{sec:fast}.
The similarity can be seen more clearly if we replace $\pi z$ by $z$
in~(\ref{eq:coth-zeta}), giving
\begin{equation}
\frac{z}{\tanh(z)} = -2 \sum_{k=0}^\infty
(-1)^{k}\frac{\zeta(2k)}{\pi^{2k}}z^{2k}\,,
\label{eq:coth-zeta2}
\end{equation}
since it is the rational number $\zeta(2n)/\pi^{2n}$ that we need
in order to compute $B_{2n}$ from~(\ref{eq:B2k_zeta}).
In fact, it is easy to see that~(\ref{eq:coth-zeta2}) is equivalent
to~(\ref{eq:genC}).

There is a vast literature on Bernoulli, Tangent and Secant numbers.  For 
example, the bibliography of Dilcher and Slavutskii~\cite{Dilcher-biblio}
contains more than $2000$ items. Thus, we do not attempt to give a complete
list of references to related work. However, we briefly mention the problem
of computing {\em irregular primes}~\cite{BCEMS,Buhler-Harvey},
which are odd primes $p$ such that $p$ divides the class number of
the $p$-th cyclotomic field. The algorithms that we present in
\S\S\ref{sec:fast}--\ref{sec:faster} below are not suitable for this task
because they take too much memory. It is much more space-efficient to use
a modular algorithm where the computations are performed modulo a single
prime (or maybe the product of a small number of primes), as
in~\cite{BCEMS,Buhler-Harvey,Crandall-Pomerance,Harvey10a}.  
Space can also be saved by
the technique of ``multisectioning'', which is described by
Crandall~\cite[\S3.2]{Crandall} and Hare~\cite{Hare}.

\section{Tangent and Secant Numbers}
\label{sec:tangent}
The {\em Tangent numbers} $T_n\;\; (n >0)$ (also called {\em Zag} numbers)
are defined by
\begin{equation*} 
\sum_{n > 0} T_n\, \frac{z^{2n-1}}{(2n-1)!} = \tan z
= \frac{\sin z}{\cos z}\,\raisedot
\end{equation*}
Similarly, the {\em Secant numbers} $S_n\;\; (n \ge 0)$
(also called {\em Euler} or {\em Zig} numbers) are defined by
\begin{equation*} 
\sum_{n \ge 0} S_n\, \frac{z^{2n}}{(2n)!} = \sec z
= \frac{1}{\cos z}\,\raisedot
\end{equation*}
Unlike the Bernoulli numbers, the Tangent and Secant numbers are 
positive integers.
Because $\tan z$ and $\sec z$ have poles at $z = \pi/2$, we expect
$T_n$ to grow roughly like $(2n-1)!(2/\pi)^n$
and $S_n$ like $(2n)!(2/\pi)^n$. 
To obtain more precise estimates, let 
\begin{equation*} \zeta_0(s) = (1 - 2^{-s})\zeta(s) = 1 + 3^{-s} + 5^{-s} + \cdots\end{equation*}
be the {\em odd} zeta-function. Then
\begin{equation}
\frac{T_n}{(2n-1)!} = \frac{2^{2n+1}\zeta_0(2n)}{\pi^{2n}}
 \sim \frac{2^{2n+1}}{\pi^{2n}}
\label{eq:Tk_odd_zeta}
\end{equation}
(this can be proved in the same way as Euler's relation (\ref{eq:B2k_zeta})
for the Bernoulli numbers).
We also have~\cite[(23.2.22)]{AS}
\begin{equation} 
\frac{S_n}{(2n)!} = \frac{2^{2n+2}\beta(2n+1)}{\pi^{2n+1}}
 \sim \frac{2^{2n+2}}{\pi^{2n+1}}\,\raisecomma
\label{eq:Sk_beta}
\end{equation}
where
\begin{equation}
\beta(s) = \sum_{j=0}^\infty (-1)^j (2j+1)^{-s}.
\label{eq:beta}
\end{equation}

From (\ref{eq:B2k_zeta}) and (\ref{eq:Tk_odd_zeta}), we see that
\begin{equation}
T_n = (-1)^{n-1}2^{2n}(2^{2n}-1)\frac{B_{2n}}{2n}\,\raisedot
\label{eq:TkB2k}
\end{equation}
This can also be proved directly, without involving the zeta-function, by using
the identity
\begin{equation*}
\tan z = \frac{1}{\tan z} - \frac{2}{\tan(2z)}\,\raisedot
\end{equation*}

Since $T_n \in {\mathbb{Z}}$, it follows from (\ref{eq:TkB2k}) 
that the odd primes in the denominator of $B_{2n}$
must divide $2^{2n}-1$. This is compatible with the Von Staudt--Clausen
theorem, since $(p-1)|2n$ implies $p|(2^{2n}-1)$ by Fermat's little theorem.

$T_n$ has about $4n$ more bits than $\lceil B_{2n} \rceil$,
but both have $2n \lg n + O(n)$ bits, so asymptotically there is
not much difference between the sizes of $T_n$ and $\lceil B_{2n} \rceil$.
Thus, if our aim is to compute $B_{2n}$, we do not
lose much by first computing $T_n$, and this may be more convenient
since $T_n \in {\mathbb{Z}}$, $B_{2n} \in {\mathbb{Q}}$.

\section{A Fast Algorithm for Bernoulli Numbers}
\label{sec:fast}
Harvey~\cite{Harvey10a} showed how $B_n$ could be computed exactly,
using a modular algorithm and the Chinese remainder theorem, in
$O(n^2(\log n)^{2 + o(1)})$ bit-operations. The same complexity can be 
obtained using (\ref{eq:B2k_zeta}) and the Euler product for the zeta-function
(see the discussion in Harvey~\cite[\S1]{Harvey10a}).

In this section we show how to compute 
{\em all} of $B_0, \ldots, B_n$ with almost the same
complexity bound (only larger by a factor $O(\log n)$).
In \S\ref{sec:faster} we give an even faster algorithm, which
avoids the $O(\log n)$ factor.

Let $A(z) = a_0 + a_1z + a_2z^2 + \cdots$ be a power series with coefficients
in ${\mathbb R}$, with $a_0 \ne 0$.
Let $B(z) = b_0 + b_1z + \cdots$ be the {\em reciprocal} power series,
so $A(z)B(z) = 1$.
Using the FFT, we
can multiply polynomials of degree $n-1$ with
$O(n \log n)$ real operations. 
Using Newton's method~\cite{Kung,Sieveking}, 
we can compute $b_0, \ldots, b_{n-1}$
with the {\em same} complexity $O(n \log n)$,
up to a constant factor.

Taking $A(z) = (\exp(z) - 1)/z$ and working with 
$N$-bit floating-point numbers, where $N = n\lg(n) + O(n)$,
we get $B_0, \ldots, B_n$ to sufficient accuracy
to deduce the exact (rational) result. 
(Alternatively, use~(\ref{eq:genC}) to avoid computing the 
terms with odd subscripts, since these vanish except for $B_1$.)
The work involved is $O(n \log n)$
floating-point operations, each of which can be done with $N$-bit accuracy
in $O(n (\log n)^{2+o(1)})$ bit-operations.
Thus, overall we get $B_0,\ldots, B_n$ with $O(n^2 (\log n)^{3 + o(1)})$
bit-operations. Similarly for Secant and Tangent numbers.
We omit a precise specification of $N$ and a detailed error analysis of 
the algorithm, since it is improved in the following section.

\section{A Faster Algorithm for Tangent and Bernoulli Numbers}
\label{sec:faster}
To improve the algorithm of \S\ref{sec:fast} 
for Bernoulli numbers, we use the
``Kronecker--Sch\"onhage trick''~\cite[\S1.9]{MCA}. 
Instead of working with power series $A(z)$ (or polynomials, which can be
regarded as truncated power series), we work with binary numbers 
$A(z)$ where $z$ is a suitable (negative) power of $2$.

The idea is to compute a single real number ${\cal A}$ which is defined in
such a way that the numbers that we want to compute are encoded in the
binary representation of~${\cal A}$.  For example, consider the series
\[\sum_{k>0} k^2 z^k = \frac{z(1+z)}{(1-z)^3}, \;\; |z| < 1.\] 
The right-hand side is an easily-computed rational function of $z$, say $A(z)$.  
We use decimal rather than binary for expository purposes.
With $z = 10^{-3}$ we easily find
\[A(10^{-3}) = \frac{1001000}{997002999} 
	     = 0.00\underbar{1}\,00\underbar{4}\,00\underbar{9}\,%
0\underbar{16}\,0\underbar{25}\,0\underbar{36}\,0\underbar{49}\,%
0\underbar{64}\,0\underbar{81}\,\underbar{100}\,\cdots\]
Thus, if we are interested in the finite sequence of squares 
$(1^2, 2^2, 3^2, \ldots, 10^2)$,
it is sufficient to compute ${\cal A} = A(10^{-3})$ 
correctly rounded to $30$ decimal
places, and we can then ``read off'' the squares from the decimal
representation of ${\cal A}$.

Of course, this example is purely for illustrative purposes, because it is
easy to compute the sequence of squares directly.  
However, we use the same
idea to compute Tangent numbers. Suppose we want the first $n$ Tangent
numbers $(T_1, T_2, \ldots, T_n)$.  The generating function
\[\tan z = \sum_{k \ge 1} T_k\, \frac{z^{2k-1}}{(2k-1)!} \]
gives us almost what we need, but not quite, because the coefficients
are rationals, not integers.  Instead, consider
\begin{equation}
(2n-1)!\tan z = \sum_{k=1}^n T'_{k,n}\, z^{2k-1} + R_n(z),
\label{eq:tsum}
\end{equation}
where
\begin{equation}
T'_{k,n} = \frac{(2n-1)!}{(2k-1)!}\,T_k			
\label{eq:Tprime}
\end{equation}
is an integer for $1 \le k \le n$,
and
\begin{equation}
R_n(z) = \sum_{k=n+1}^\infty T'_{k,n}\, z^{2k-1} =
  (2n-1)!\sum_{k=n+1}^\infty T_k\, \frac{z^{2k-1}}{(2k-1)!} 	\label{eq:Rn}
\end{equation}
is a remainder term which is small if $z$ is sufficiently small.
Thus, choosing\linebreak 
$z = 2^{-p}$ with $p$ sufficiently large, the first $2np$
binary places of $(2n-1)!\tan z$ define $T'_{1,n}, T'_{2,n},\ldots,T'_{n,n}$.
Once we have computed $T'_{1,n}, T'_{2,n}, \ldots, T'_{n,n}$ it is easy to
deduce $T_1, T_2, \ldots, T_n$ from
\[T_k = \frac{T'_{k,n}}{(2n-1)!/(2k-1)!}\,\raisedot\]

For this idea to work, two conditions must be satisfied.  First, we need
\begin{equation}
0 \le T'_{k,n} < {1}/{z^2} = 2^{2p}, \;\; 1 \le k \le n, \label{eq:condA}
\end{equation}
so we can read off the $T'_{k,n}$ from the binary representation
of $(2n-1)!\tan z\,$.  Since we have a good asymptotic estimate for $T_k$,
it is not hard to choose $p$ sufficiently large for this condition to hold.

Second, we need the remainder term $R_n(z)$ to be sufficiently small that it
does not influence the estimation of $T'_{n,n}$.
A sufficient condition is
\begin{equation}
0 \le R_n(z) < z^{2n-1}.			\label{eq:condB}
\end{equation}
Choosing $z$ sufficiently small (i.e.\ $p$ sufficiently large) 
guarantees that 
condition~(\ref{eq:condB}) holds, since $R_n(z)$ is $O(z^{2n+1})$ as
$z \to 0$ with $n$ fixed.

Lemmas \ref{lemma:3} and \ref{lemma:4} below
give sufficient conditions for
(\ref{eq:condA}) and (\ref{eq:condB}) to hold.

\pagebreak[3]
\begin{lemma}
\label{lemma:1}
\[\frac{T_k}{(2k-1)!} \le \left(\frac{2}{\pi}\right)^{2(k-1)}
  \;\; {\rm for} \;\; k \ge 1.\]
\end{lemma}
\begin{proof}
From (\ref{eq:Tk_odd_zeta}),
\[\frac{T_k}{(2k-1)!} = 2\left(\frac{2}{\pi}\right)^{2k}\zeta_0(2k)
\le 2\left(\frac{2}{\pi}\right)^{2k}\zeta_0(2) 
\le \left(\frac{2}{\pi}\right)^{2k}\frac{\pi^2}{4}
= \left(\frac{2}{\pi}\right)^{2(k-1)} \qed\]
\end{proof}

\begin{lemma}
\label{lemma:2}
$(2n-1)! \le n^{2n-1}$ for $n \ge 1$. % Equality iff $n = 1$.
\end{lemma}
\begin{proof}
\[(2n-1)! = n\; \prod_{j=1}^{n-1} (n-j)(n+j) 
= n\; \prod_{j=1}^{n-1} (n^2 - j^2) \le n^{2n-1}\]
with equality iff $n = 1$. 	\qed
\end{proof}

\begin{lemma}
\label{lemma:3}
If $k \ge 1$, $n \ge 2$, $p = \lceil n \lg(n) \rceil$, $z = 2^{-p}$, 
and $T'_{k,n}$ is as in {\rm (\ref{eq:Tprime}),} 
then $z \le n^{-n}$ and $T'_{k,n} < 1/z^2$.
\end{lemma}
\begin{proof}
We have $z = 2^{-p} = 2^{-\lceil n \lg(n) \rceil} \le 
2^{-n \lg(n)} = n^{-n}$, which proves the first part of the Lemma.

Assume $k \ge 1$ and $n \ge 2$.
From Lemma~\ref{lemma:1}, we have 
\[T'_{k,n} \le (2n-1)!\left(\frac{2}{\pi}\right)^{2(k-1)} 
\le (2n-1)!,\]
and from Lemma~\ref{lemma:2} it follows that
\[T'_{k,n} \le n^{2n-1} < n^{2n}.\]
From the first part of the Lemma, $n^{2n} \le 1/z^2$, so the second part
follows. \qed
\end{proof}

\begin{lemma}
\label{lemma:4}
If $n \ge 2$, $p = \lceil n \lg(n) \rceil$, $z = 2^{-p}$, 
and $R_n(z)$ is as defined in {\rm (\ref{eq:Rn}),}
then\linebreak $0 < R_n(z) < 0.1\,z^{2n-1}$\,.
\end{lemma}
\begin{proof}
Since all the terms in the sum defining $R_n(z)$ are positive, it is
immediate that $R_n(z) > 0$.

\pagebreak[3]
Since $n \ge 2$, we have $p \ge 2$ and $z \le 1/4$.  
Now, using Lemma~\ref{lemma:1},
\begin{eqnarray*}
R_n(z) 	&=& \sum_{k=n+1}^\infty T'_{k,n}z^{2k-1}\\
	&\le& (2n-1)!\,\sum_{k=n+1}^\infty 
		\left(\frac{2}{\pi}\right)^{2(k-1)}z^{2k-1}\\
	&\le& (2n-1)!\,\left(\frac{2}{\pi}\right)^{2n}z^{2n+1}
	  \left(1 + \left(\frac{2z}{\pi}\right)^2 +  
		    \left(\frac{2z}{\pi}\right)^4 + \cdots \right)\\
	&\le& (2n-1)!\,\left(\frac{2}{\pi}\right)^{2n}z^{2n+1}
		\left/\left(1 - \left(\frac{2z}{\pi}\right)^2\right)\right.\,.
\end{eqnarray*}
Since $z \le 1/4$, we have $1/(1 - (2z/\pi)^2) < 1.026$.  
Also, from Lemma~\ref{lemma:2},\\ 
$(2n-1)! \le n^{2n-1}$. Thus, we have
\begin{equation*}
\frac{R_n(z)}{z^{2n-1}} < 1.026\,n^{2n-1}\left(\frac{2}{\pi}\right)^{2n}{z^2}.
\end{equation*}
Now $z^2 \le n^{-2n}$ from
the first part of Lemma \ref{lemma:3}, so
\begin{equation}
\frac{R_n(z)}{z^{2n-1}} < \frac{1.026}{n}\left(\frac{2}{\pi}\right)^{2n}.
\label{eq:Rnzineq}
\end{equation}
The right-hand side is a monotonic decreasing function of $n$, so is bounded
above by its value when $n=2$, giving $R_n(z)/z^{2n-1} < 0.1\,$.
\qed
\end{proof}

A high-level description of the resulting Algorithm FastTangentNumbers
is given in Figure~\ref{fig:FastTangentNumbers}.
The algorithm computes the Tangent numbers
$T_1,T_2,\ldots,T_n$ using the Kronecker-Sch\"onhage trick as described above,
and deduces the Bernoulli numbers
$B_2,B_4,\ldots,B_{2n}$ from the relation~(\ref{eq:TkB2k}).
\begin{figure}[h]
\begin{tabbing}
====\====\====\====\====\kill
\>{\bf Input:} integer $n \ge 2$\\
\>{\bf Output:} Tangent numbers $T_1, \ldots, T_n$
	and (optional) Bernoulli numbers $B_2, B_4, \ldots, B_{2n}$\\
\>\> $p \leftarrow \lceil n\lg(n)\rceil$\\
\>\> $z \leftarrow 2^{-p}$\\
\>\> $S \leftarrow \sum_{0 \le k < n}(-1)^{k} z^{2k+1} 
        \times (2n)!/(2k+1)!$\\[2pt]
\>\> $C \leftarrow \sum_{0 \le k < n}(-1)^k z^{2k} \times (2n)!/(2k)!$\\[2pt]
\>\> $V \leftarrow \lfloor z^{1-2n} \times (2n-1)! \times S/C \rceil$ 
 (here $\lfloor x \rceil$ means round $x$ to nearest integer)\\
\>\> Extract $T'_{k,n} = T_k(2n-1)!/(2k-1)!$, $1 \le k \le n$, 
from the binary representation of $V$\\
\>\> $T_k \leftarrow T'_{k,n} \times (2k-1)! / (2n-1)!$, $k = n, n-1, \ldots, 1$\\
\>\> $B_{2k} \leftarrow (-1)^{k-1}(k \times T_k / 2^{2k-1}) / (2^{2k}-1)$,
$k = 1, 2, \ldots, n$ (optional)\\
\>\> {\bf return} $T_1, T_2, \ldots, T_n$ and (optional)
$B_2, B_4, \ldots, B_{2n}$\\
\end{tabbing}
\caption{Algorithm FastTangentNumbers 
(also optionally computes Bernoulli numbers)}
\label{fig:FastTangentNumbers}
\end{figure}

\pagebreak[3]
In order to achieve the best complexity, 
the algorithm must be implemented
carefully using binary arithmetic. 
The computations of $S$ (an approximation to $(2n)!\sin z$)
and $C$ (an approximation to $(2n)!\cos z$) involve
computing ratios of factorials such as $(2n)!/(2k)!$, where
$0 \le k \le n$. This can be done in time
$O(n^2 (\log n)^2)$ by a straightforward algorithm.
The $N$-bit division to compute $S/C$ (an approximation to $\tan z$)
can be done in time $O(N\log(N)\log\log(N))$ 
by the Sch\"onhage--Strassen algorithm combined with 
Newton's method~\cite[\S4.2.2]{MCA}. Here it is sufficient to take
$N = 2np + 2 = 2n^2\lg(n) + O(n)$.
Note that 
\begin{equation}
V = \sum_{k=1}^n 2^{2(n-k)p} T'_{k,n} 		\label{eq:Vexact}
\end{equation}
is just the finite sum in (\ref{eq:tsum}) scaled by $z^{1-2n}$
(a power of two),
and the integers $T'_{k,n}$ can simply be ``read off'' from the binary
representation of~$V$ in $n$ blocks of $2p$ consecutive bits.
The $T'_{k,n}$ can then be scaled by ratios of factorials
in time $O(n^2 (\log n)^{2+o(1)})$
to give the Tangent numbers $T_1, T_2, \ldots, T_n$.

The correctness of the computed Tangent numbers follows from
Lemmas~\ref{lemma:3}--\ref{lemma:4}, apart from possible errors
introduced by $S/C$ being only an approximation
to $\tan(z)$. Lemma~\ref{lemma:5} 
shows that this error is sufficiently small.
\begin{lemma}
Suppose that $n \ge 2$, $z$, $S$ and $C$ as in Algorithm
FastTangentNumbers.  Then
\begin{eqnarray}
z^{1-2n}(2n-1)!\,\left|\frac{S}{C} - \tan z\right| &<& 0.02\,. 
 \label{eq:L5g}
\end{eqnarray}
\label{lemma:5}
\end{lemma}
\begin{proof}
We use the inequality
\begin{equation}
\left|\frac{A}{B} - \frac{A'}{B'}\right| \le
\frac{|A|\cdot|B-B'| + |B|\cdot|A-A'|}{|B|\cdot|B'|}\,\raisedot
\label{eq:genineq}
\end{equation}
Take $A = \sin z$, $B = \cos z$, $A' = S/(2n)!$, $B' = C/(2n)!$
in~(\ref{eq:genineq}). Since $n \ge 2$ we have $0 < z \le 1/4$.
Then $|A| = |\sin z| < z$. Also, $|B| = |\cos z| > 31/32$ from the
Taylor series $\cos z = 1 - z^2/2 + \cdots$, 
which has terms of alternating sign
and decreasing magnitude. 
By similar arguments, $|B'| \ge 31/32$,
$|B-B'| < z^{2n}/(2n)!$,
and\\ $|A-A'| < z^{2n+1}/(2n+1)!$.
Combining these inequalities and using~(\ref{eq:genineq}), we obtain
\[
\left|\frac{S}{C} - \tan z\right| < 
\frac{6\cdot 32\cdot 32}{5 \cdot 31 \cdot 31}\frac{z^{2n+1}}{(2n)!} <
\frac{1.28\,z^{2n+1}}{(2n)!}\,\raisedot
\]
Multiplying both sides by $z^{1-2n}(2n-1)!$ and using 
$1.28\,z^2/(2n) \le 0.02$,
we obtain the inequality~(\ref{eq:L5g}).
This completes the proof of Lemma~\ref{lemma:5}.
\qed
\end{proof}

In view of the constant $0.02$ in (\ref{eq:L5g})
and the constant $0.1$ in Lemma~\ref{lemma:4},
the effect of all sources of error in computing 
$z^{1-2n}(2n-1)!\tan z$ is at most $0.12 < 1/2$, which is too small to
change the computed integer $V$, 
that is to say, the computed $V$
is indeed given by~(\ref{eq:Vexact}).

The computation of the Bernoulli numbers
$B_2, B_4, \ldots, B_{2n}$ from $T_1, \ldots, T_n$, is straightforward
(details depending on exactly how rational numbers are to be represented).
The entire computation takes time
\begin{equation*} 
O(N(\log N)^{1+o(1)}) = O(n^2(\log n)^{2+o(1)}).
\end{equation*} 
Thus, we have proved:
\begin{theorem}
The Tangent numbers $T_1, \ldots, T_n$ and
Bernoulli numbers $B_2, B_4, \ldots, B_{2n}$ 
can be computed in $O(n^2(\log n)^{2+o(1)})$ bit-operations
using $O(n^2 \log n)$ space.
\label{thm:tangent}
\end{theorem}
A small modification of the above can be used to compute
the Secant numbers $S_0, S_1, \ldots, S_n$ in
$O(n^2(\log n)^{2+o(1)})$ bit-operations and $O(n^2 \log n)$ space.
The bound on Tangent 
numbers given by Lemma~\ref{lemma:1} can be replaced by the bound
\[\frac{S_n}{(2n)!} \le 2\left(\frac{2}{\pi}\right)^{2n+1}\]
which follows from~(\ref{eq:Sk_beta}) since $\beta(2n+1) < 1$. 

We remark that an efficient implementation of Algorithm
FastTangentNumbers in a high-level language such as 
Sage~\cite{Sage} or Magma~\cite{Magma}
is nontrivial, because it requires access to the internal binary
representation of high-precision integers. 
Everything can be done using (implicitly scaled)
integer arithmetic~-- there is no need for floating-point~-- but for
the sake of clarity we did not
include the scaling in Figure~\ref{fig:FastTangentNumbers}.
If floating-point arithmetic is used, a precision of $N$ bits is sufficient,
where $N = 2np+2$.

Comparing our Algorithm FastTangentNumbers with Harvey's modular
algorithm~\cite{Harvey10a}, we see that there is a space-time trade-off:
Harvey's algorithm uses less space (by a factor of order $n$)
to compute a single $B_n$,
but more time (again by a factor of order $n$) to compute all of
$B_1, \ldots, B_n$. 
Harvey's algorithm has better locality and is readily parallelisable.

In the following section we give much simpler algorithms which are
fast enough for most practical purposes, and are based on three-term
recurrence relations.

\section{Algorithms Based On Three-Term Recurrences}
\label{sec:three_term}
Akiyama and Tanigawa~\cite{Kaneko00}
gave an algorithm for computing Bernoulli numbers based
on a three-term recurrence.  However, it is only useful for exact
computations, since it is numerically unstable if applied using
floating-point arithmetic. It is faster to use a stable
recurrence for computing Tangent numbers, and then deduce the 
Bernoulli numbers from~(\ref{eq:TkB2k}).

\subsection{Bernoulli and Tangent numbers}
We now give a {stable} three-term recurrence and corresponding in-place
algorithm for computing Tangent numbers.  
The algorithm is perfectly stable since
all operations are on positive integers and there is no cancellation.
Also, it involves less arithmetic than the Akiyama-Tanigawa algorithm.
This is partly because the operations are on integers rather than
rationals, and partly because there are fewer operations since we take
advantage of zeros.

Bernoulli numbers can be computed using Algorithm TangentNumbers
and the relation~(\ref{eq:TkB2k}). The time required for the 
application of~(\ref{eq:TkB2k}) is negligible.

The recurrence (\ref{MCA463}) that we use 
was given by Buckholtz and Knuth~\cite{KB67},
but they did not give our in-place Algorithm TangentNumbers explicitly.
Related\linebreak 
recurrences with applications to parallel computation
were considered by Hare~\cite{Hare}.

Write $t = \tan x$, $D = {\rm d}/{\rm d}x$, so $Dt = 1 + t^2$
and $D(t^n) = nt^{n-1}(1+t^2)$ for $n \ge 1$.
It is clear that $D^nt$ is a polynomial in $t$, say $P_n(t)$.
Write $P_n(t) = \sum_{j \ge 0} p_{n,j}t^j$.
Then ${\rm deg}(P_n) = n+1$ and, from the formula for $D(t^n)$,
\begin{equation}
p_{n,j} = (j-1)p_{n-1,j-1} + (j+1)p_{n-1,j+1}.	\label{MCA463}
\end{equation} % (4.63)
We are interested in 
$T_k = ({\rm d}/{\rm d}x)^{2k-1}\tan x\,\vert_{x=0} 
     = P_{2k-1}(0) = p_{2k-1,0}$, 
which can be computed from the
recurrence (\ref{MCA463}) in $O(k^2)$ operations 
using the obvious boundary conditions.
We save work by noticing that $p_{n,j} = 0$ if $n+j$ is even.
The resulting algorithm is given in
Figure~\ref{fig:TangentNumbers}.

\begin{figure}[h]
\begin{tabbing}
====\====\====\====\====\kill
\>{\bf Input:} positive integer $n$\\
\>{\bf Output:} Tangent numbers $T_1, \ldots, T_n$\\
\>\> $T_1 \leftarrow 1$\\
\>\> {\bf for} {$k$ {\bf from} $2$ {\bf to} $n$}\\
\>\>\>  $T_k \leftarrow (k-1) T_{k-1}$\\
\>\> {\bf for} {$k$ {\bf from} $2$ {\bf to} $n$}\\
\>\>\> {\bf for} {$j$ {\bf from} $k$ {\bf to} $n$}\\
\>\>\>\> $T_j \leftarrow (j-k) T_{j-1} + (j-k+2) T_j$\\ 
\>\> {\bf return} $T_1, T_2, \ldots, T_n$.\\
\end{tabbing}
\caption{Algorithm TangentNumbers}
\label{fig:TangentNumbers}
\end{figure}

The first {\bf for} loop initializes $T_k = p_{k-1,k} = (k-1)!$. The variable
$T_k$ is then used to store $p_{k,k-1}$, $p_{k+1,k-2}$, $\ldots$, 
$p_{2k-2,1}$, $p_{2k-1,0}$ at successive iterations of the second {\bf for}
loop.
Thus, when the algorithm terminates, $T_k = p_{2k-1,0}$, as expected.

\pagebreak[3]
The process in the case $n=3$ 
is illustrated in Figure~\ref{fig:process}, where $T_k^{(m)}$
denotes the value of the variable $T_k$ at successive iterations
$m = 1, 2, \ldots, n$. It is instructive to compare a similar Figure for the 
Akiyama-Tanigawa algorithm in~\cite{Kaneko00}.
\begin{figure}[h]
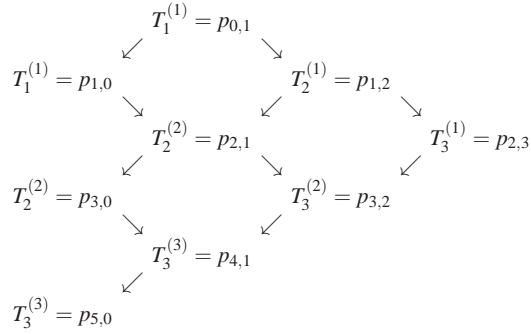

\begin{equation*} 
\begin{array}{ccccccc}
&& T_1^{(1)} = p_{0,1}\\
& \swarrow && \searrow\\
T_1^{(1)} = p_{1,0} &&&& T_2^{(1)} = p_{1,2}\\
& \searrow && \swarrow && \searrow\\
&& T_2^{(2)} = p_{2,1} &&&& T_3^{(1)} = p_{2,3}\\	
& \swarrow && \searrow && \swarrow\\
T_2^{(2)} = p_{3,0} &&&& T_3^{(2)} = p_{3,2}\\
& \searrow && \swarrow\\
&& T_3^{(3)} = p_{4,1}\\
& \swarrow\\
T_3^{(3)} = p_{5,0}\\
\end{array}
\end{equation*}
\caption{Dataflow in Algorithm TangentNumbers for $n = 3$}
\label{fig:process}
\end{figure}

Algorithm TangentNumbers takes $\Theta(n^2)$ operations on positive
integers.  The integers $T_n$ have $O(n \log n)$ bits, other
integers have $O(\log n)$ bits.
Thus, the overall complexity is 
$O(n^3 (\log n)^{1+o(1)})$ bit-operations,
or $O(n^3 \log n)$ word-operations if $n$ fits in a single word.

The algorithm is not optimal, but it is good in practice 
for moderate values of $n$, and much simpler than 
asymptotically faster algorithms such as those described in 
\S\S\ref{sec:fast}--\ref{sec:faster}.
For example, using a straightforward Magma implementation of 
Algorithm TangentNumbers,
we computed the first $1000$ Tangent numbers in 1.50 sec
on a 2.26 GHz Intel Core 2 Duo.
For comparison, it takes 1.92 sec for a single $N$-bit
division computing $T$ in Algorithm FastTangentNumbers
(where $N = 19931568$ corresponds to $n=1000$).
Thus, we expect the crossover point where Algorithm FastTangentNumbers
actually becomes faster to be slightly larger than $n=1000$
(but dependent on implementation details).

\subsection{Secant numbers}

A similar algorithm may be used to compute Secant numbers.
Let $s = \sec x$, $t = \tan x$,
and $D = {\rm d}/{\rm d}x$. Then $Ds = st$,
$D^2 s = s(1 + 2t^2)$, and in general
$D^n s = sQ_n(t)$, where $Q_n(t)$ is a polynomial of degree~$n$ in $t$.
The Secant numbers are given by $S_k = Q_{2k}(0)$.
Let $Q_n(t) = \sum_{k \ge 0} q_{n,k}t^k$. From
\[D(st^k) = st^{k+1} + kst^{k-1}(1+t^2)\] 
we obtain the three-term recurrence
\begin{equation}
q_{n+1,k} = kq_{n,k-1} + (k+1)q_{n,k+1} \;\;{\rm for}\;\; 1 \le k \le n.
\label{eq:secant3}
\end{equation}
By avoiding the computation of terms $q_{n,k}$ that are known to be zero
($n+k$ odd),
and ordering the computation in a manner analogous to that used
for Algorithm TangentNumbers,
we obtain Algorithm SecantNumbers (see Figure~\ref{fig:SecantNumbers}),
which computes 
the Secant numbers in place using non-negative integer arithmetic.
\begin{figure}[ht]
\begin{tabbing}
====\====\====\====\====\kill
\>{\bf Input:} positive integer $n$\\
\>{\bf Output:} Secant numbers $S_0, S_1 \ldots, S_n$\\
\>\> $S_0 \leftarrow 1$\\
\>\> {\bf for} {$k$ {\bf from} $1$ {\bf to} $n$}\\
\>\>\>  $S_k \leftarrow k S_{k-1}$\\
\>\> {\bf for} {$k$ {\bf from} $1$ {\bf to} $n$}\\
\>\>\> {\bf for} {$j$ {\bf from} $k+1$ {\bf to} $n$}\\
\>\>\>\> $S_j \leftarrow (j-k) S_{j-1} + (j-k+1) S_j$\\ 
\>\> {\bf return} $S_0, S_1, \ldots, S_n$.\\
\end{tabbing}
\caption{Algorithm SecantNumbers}
\label{fig:SecantNumbers}
\end{figure}

\subsection{Comparison with Atkinson's algorithm}

Atkinson~\cite{Atkinson} gave an elegant algorithm for computing both
the Tangent numbers $T_1, T_2, \ldots, T_n$ and
the Secant numbers $S_0, S_1, \ldots, S_n$ using a ``Pascal's triangle''
style of algorithm that only involves additions of non-negative integers.
Since a triangle with $2n+1$ rows in involved, Atkinson's algorithm
requires $2n^2 + O(n)$ integer additions.  This can be compared with
$n^2/2 + O(n)$ additions and $n^2 + O(n)$ multiplications (by small 
integers) for our Algorithm TangentNumbers, and similarly for
Algorithm SecantNumbers.

Thus, we might expect Atkinson's algorithm to be slower than Algorithm
TangentNumbers. Computational
experiments confirm this.  With $n = 1000$, Algorithm TangentNumbers
programmed in Magma
takes $1.50$ seconds on a 2.26 GHz Intel Core 2 Duo, 
algorithm SecantNumbers also takes $1.50$ seconds,
and
Atkinson's algorithm takes $4.51$ seconds. Thus, even if both Tangent and
Secant numbers are required, Atkinson's algorithm is slightly slower.
It also requires about twice as much memory.

\pagebreak[3]
\begin{acknowledgement}
We thank Jon Borwein
for encouraging the belief that high-precision
computations are useful in ``experimental'' mathematics~\cite{Borwein99},
e.g.\ in the PSLQ algorithm~\cite{Ferguson}.
Ben F. ``Tex'' Logan, Jr.~\cite{Logan} suggested the use of Tangent numbers
to compute Bernoulli numbers.
Christian Reinsch~\cite{Reinsch79} 
pointed out the numerical instability of
the recurrence~(\ref{eq:unstable})
and suggested the use of the numerically stable recurrence~(\ref{MCA460}).
Christopher Heckman % checkman@math.asu.edu, email of 20110803
kindly drew our attention to Atkinson's algorithm~\cite{Atkinson}.
We thank Paul Zimmermann for his comments.
Some of the material presented here
is drawn from the recent book
{\em Modern Computer Arithmetic}~\cite{MCA} 
(and as-yet-unpublished solutions to exercises in the book).
In particular, see \cite[\S4.7.2 and exercises 4.35--4.41]{MCA}.
Finally, we thank David Bailey, Richard Crandall, and two anonymous referees
for suggestions and pointers to additional references.

\end{acknowledgement}

\pagebreak[4]
\bibliographystyle{}
\bibliography{}

\end{document}